\documentclass{amsart}

\usepackage{amsmath, amsthm, amsfonts}
\usepackage{amssymb}
\usepackage[utf8]{inputenc}
\usepackage{latexsym}
\usepackage{verbatim}
\usepackage{url}
\usepackage{textcomp}
\usepackage[all,cmtip]{xy}
\usepackage{color}
\usepackage{booktabs}
\usepackage{hyperref}
\usepackage{mathtools}

\usepackage{tabularx}
\input{xypic}
\usepackage{enumerate}
\usepackage{amssymb}
\usepackage{array}
\usepackage{tikz}
\usepackage{rotating}
\usetikzlibrary{arrows,shapes,positioning}
\usetikzlibrary{decorations.markings}
\tikzstyle arrowstyle=[scale=1]

\DeclareMathAlphabet{\mathfr}{U}{euf}{m}{n}

\newtheorem{theorem}{Theorem}[section]
\newtheorem*{theorem*}{Theorem}

\newtheorem{proposition}[theorem]{Proposition}
\newtheorem{corollary}[theorem]{Corollary}

\newtheorem{lemma}[theorem]{Lemma}
\newtheorem{question}[theorem]{Question}
\newtheorem{assumption}[theorem]{Assumption}

\theoremstyle{remark}
\newtheorem{remark}[theorem]{Remark}

\newtheorem{example}[theorem]{Example}

\definecolor{myblue}{rgb}{.2,.6,.75}
\definecolor{mygreen}{rgb}{.4,.7,.4}

\newcommand\acc[2]{\ensuremath{{}^{#1}\hskip-0.3ex{#2}}}

\newcommand{\Q}{\mathbb Q}
\newcommand{\QQ}{\mathbb Q}
\newcommand{\Qbar}{{\overline{\mathbb Q}}}
\newcommand{\kbar}{{\bar{K}}}
\newcommand{\bK}{{\bar{K}}}
\newcommand{\GK}{G_K}
\newcommand{\Gal}{\mathrm{Gal}}
\newcommand{\R}{\mathbb R}
\newcommand{\Z}{\mathbb Z}

\newcommand{\C}{\mathbb C}

\newcommand{\RR}{\mathbb R}

\newcommand{\CC}{\mathbb C}

\newcommand{\GL}{\mathrm{GL}}
\newcommand{\M}{\mathrm{M }}

\newcommand{\GSp}{\mathrm{GSp}}

\newcommand{\et}{\mathrm{et}}
\newcommand{\End}{\operatorname{End}}
\newcommand{\Hom}{\operatorname{Hom}}
\newcommand{\Frob}{\operatorname{Fr}}
\newcommand{\Ind}{\operatorname{Ind}}

\newcommand{\Jac}{{\operatorname{Jac}}}
\newcommand{\Std}{{\operatorname{Std}}}

\newcommand{\Aut}{\operatorname{Aut}}

\newcommand{\OO}{\mathcal O}

\newcommand{\p}{\mathfrak{p}}

\newcommand{\PP}{\mathfrak{P}}
\newcommand{\af}{\mathfrak{a}}

\newcommand{\Nf}{\mathfrak{N}}
\newcommand{\Res}{\operatorname{Res}}
\newcommand{\Tr}{\operatorname{Tr}}

\newcommand{\cO}{\mathcal{O}}
\newcommand{\A}{\mathbb{A}}

\newcommand{\Unitary}{\operatorname{U}}
\newcommand{\Nm}{\operatorname{Nm}}

\numberwithin{equation}{section}

\usepackage{url}

\begin{document}
\title[]{Abelian threefolds with imaginary multiplication}

\author{Francesc Fit\'e}

\address{Departament de matem\`atiques i inform\`atica and Centre de recerca matem\`atica,
Universitat de Barcelona,
Gran via de les Corts Catalanes 585, 08007 Barcelona, Catalonia}
\email{ffite@ub.edu}
\urladdr{http://www.ub.edu/nt/ffite/}

\author{Pip Goodman}

\address{Departament de matem\`atiques i inform\`atica,
Universitat de Barcelona,
Gran via de les Corts Catalanes 585, 08007 Barcelona, Catalonia}
\email{pip.goodman@ub.edu}
\urladdr{https://pipgoodman.github.io/}

\begin{abstract} 
Let $A$ be an abelian threefold defined over a number field $K$ with potential multiplication by an imaginary quadratic field $M$.
Under mild assumptions on $K$, if $A$ has signature $(2,1)$ and the multiplication by $M$ is defined over $KM$, we attach to $A$ an elliptic curve defined over $K$ with potential complex multiplication by $M$, whose attached Galois representation is determined by the Hecke character associated to the determinant of the compatible system of $\lambda$-adic representations of $A$.
We deduce that if $\End^0(A_\bK)$ is an imaginary quadratic field, then it necessarily has class number $\leq [KM \colon M]$. 
\end{abstract}

\maketitle
\tableofcontents

\section{Introduction}\label{section: introduction}

Throughout this article $A$ denotes an abelian threefold defined over a number field $K$, $M$ a quadratic extension of $\QQ$, and we assume that $A$ has potential multiplication by $M$.
This means that there exists an algebra embedding
$$
\iota \colon M\hookrightarrow \End^0(A_\kbar) 
$$
sending $1 \in M$ to the identity automorphism of $A_{\kbar}$.
Here $\End^0(A_\kbar)$ stands for $\End(A_\kbar)\otimes \Q$, the so-called geometric endomorphism algebra of $A$.
We shall later show that $M$ is forced to be an imaginary quadratic field; see Lemma \ref{lemma: algebra embedding}. 

Via the embedding $\iota$, the $3$-dimensional $\kbar$-vector space of regular differentials $H^0(A_{\kbar}, \Omega^1_{A_\kbar})$ is equipped with an $M\otimes \kbar$-module structure.
Each embedding $\sigma\colon M\hookrightarrow \kbar$ induces a $\kbar$-algebra map $ M\otimes \kbar \rightarrow \kbar$ that we use to define
\begin{equation}\label{equation: signature}
r_\sigma\coloneqq\dim_{\kbar} H^0(A_{\kbar}, \Omega^1_{A_\kbar})\otimes_{M\otimes \kbar, \sigma} \kbar.
\end{equation}
Denote by $\overline\sigma$ the composition of $\sigma$ with the nontrivial automorphism of $M$, and call the formal sum $r_\sigma\cdot \sigma + r_{\overline \sigma}\cdot \overline\sigma $ the signature of $(A,\iota)$.
We sometimes simply write $(r_\sigma,r_{\overline \sigma})$ to denote the signature and without loss of generality we will assume that $r_\sigma\geq r_{\overline\sigma}$.
We shall later show that the signature is forced to be $(2,1)$ whenever $\End^0(A_\kbar) \not \simeq \M_3(M)$; see Proposition \ref{proposition: prel2}. 

Suppose that $\iota(M)\subseteq \End^0(A)$, in which case we say that $A$ has multiplication by $M$. For a rational prime $\ell$, let $V_\ell(A)$ denote the rational $\ell$-adic Tate module of $A$.
Via the embedding $\iota$, the space $V_\ell(A)$ is endowed with an $M\otimes \Q_\ell$-module structure.
For $\lambda$ a prime of $M$ above $\ell$, let $M_\lambda$ denote the completion of $M$ at $\lambda$ and $\sigma_\lambda \colon M\hookrightarrow M_\lambda$ the natural inclusion.
Write $V_\lambda(A)$ for $V_\ell(A)\otimes_{M\otimes \Q_\ell,\sigma_\lambda} M_\lambda$ where the tensor product is taken with respect to the $\QQ_\ell$-algebra map $M\otimes \QQ_\ell \rightarrow M_\lambda$ induced by $\sigma_\lambda$.
It is a 3-dimensional $M_\lambda$-vector space equipped with an action of the absolute Galois group $G_K \coloneqq\Gal(\kbar/K)$.
We will denote by $\overline \cdot$ the the nontrivial automorphism of $M$. Let $\varphi_\lambda:M_{\overline \lambda} \rightarrow M_{ \lambda}$ denote the $\Q_\ell$-isomorphism rendering the diagram
\begin{align*}
  \xymatrix{
M_{\overline \lambda}  \ar[r]^{\varphi_{ \lambda}} & M_{ \lambda}\\
M \ar[r]^{\overline \cdot }\ar[u]^{\sigma_{\overline \lambda} } & M  \ar[u]_{\sigma_{ \lambda }}}
\end{align*}
commutative. 
We will write $\overline V_{ \lambda}(A)$ to denote $V_{\overline \lambda}(A)$ viewed as an $M_\lambda[G_K]$-module via the isomorphism $\varphi_\lambda$. 

\textbf{Main result.} The following is the main result of this article.

\begin{theorem}\label{theorem: maintheorem}
Let $M$ be a quadratic field and $A$ an abelian threefold defined over a number field $K$ such that there exists an algebra embedding $\iota\colon M\hookrightarrow \End^0(A_{KM})$ sending $1 \in M$ to the identity automorphism of $A_{KM}$.
Suppose that $(A,\iota)$ has signature $(2,1)$. If $K\not = KM$, assume further that there exists a field embedding $\sigma\colon K \hookrightarrow \C$ such that either $\sigma(K)\subseteq \R$ or $j(\cO_M)\in \sigma(K)$.

Then there exists an elliptic curve $E$ defined over $K$ with potential complex multiplication by $M$ such that, for every prime $\lambda$ of $M$ over a rational prime $\ell$, the following is satisfied:
\begin{enumerate}[i)] 
\item If $K=KM$, then 
$$
\big(\wedge ^3V_\lambda(A)\oplus \wedge ^3\overline V_{\lambda}(A) \big)(-1)\simeq V_\ell(E)\otimes_{\Q_\ell}M_\lambda
$$
as $M_\lambda[G_K]$-modules.
\item If $K\not=KM$, then 
$$
\Ind^K_{KM}\big(\wedge ^3V_\lambda(A_{KM}) \big)(-1)\simeq V_\ell(E)\otimes_{\Q_\ell}M_\lambda
$$
as $M_\lambda[G_K]$-modules.
\end{enumerate}
In particular, $M$ has class number $\leq [KM:M]$.
\end{theorem}

See \cite[Corollary 4.9, Proposition 4.14 and Remark 4.3]{LS24} for an analogue of this theorem  in the context of variations of Hodge structures and Chow motives over $\mathbb C$ when $A$ is the Jacobian of a Picard curve. It seems interesting to explore the relation of this theorem to \cite[Theorem A]{ACV20}.

\textbf{Applications.}
Suppose now that $\End^0(A_\bK)$ is a field.
Under this assumption, we shall later show that $\iota(M)\subseteq \End^0(A_{KM})$; see Proposition \ref{proposition: prel1}.
In view of this and Proposition \ref{proposition: prel2}, base changing to $KM$ and applying Theorem \ref{theorem: maintheorem}  one obtains the following.

\begin{corollary}\label{corollary: bounded class number}
Let $A$ be an abelian threefold defined over $K$ such that $\End^0(A_\bK)$ is a field containing $M$. Then $M$ has class number $\leq [KM: M]$.
\end{corollary}

One formulation of Coleman's conjecture predicts that there exist only finitely many geometric endomorphism algebras of abelian varieties of bounded dimension defined over a number field of bounded degree (see \cite{OSZ21} for a discussion of variants of this conjecture and for its connections with several other conjectures).

Orr and Skorobogatov \cite{OS18} have shown that Coleman's conjecture holds for abelian varieties with complex multiplication.
In particular, there are only finitely many degree six fields that may occur as geometric endomorphism algebras of abelian threefolds defined over a number field of bounded degree.
When $K=\Q$, a complete list of such fields is given by \cite[Theorem 4.3.1]{Kil16}. When $K=\Q$ and $\End^0(A_\bK)$ is a sextic cyclic field, Corollary \ref{corollary: bounded class number} recovers \cite[Proposition 3.3.3]{Kil16}, which is an intermediate result towards the proof of \cite[Theorem 4.3.1]{Kil16}.

Heilbronn's class number theorem \cite{Heilbronn} states that there are only finitely many  imaginary quadratic fields of bounded class number.
In view of this, Corollary \ref{corollary: bounded class number} implies that there are only finitely many quadratic fields that occur as geometric endomorphism algebras of abelian threefolds defined over a number field of bounded degree.

Ari Shnidman pointed out to us that one may attempt an alternative approach to proving Corollary \ref{corollary: bounded class number} using the theory of canonical models of Shimura varieties \cite{Del79}. The Galois action on the irreducible components of one such model for a Shimura variety associated to the unitary group over $M$ with signature $(2,1)$ (as described in \cite{KR14}, for example) should allow one to show that the Hilbert class field of $M$ is a subfield of $KM$, recovering the bound $h_M\leq [KM:M]$.

We also expect Theorem \ref{theorem: maintheorem} to have applications in computing the Zeta functions of abelian threefolds with imaginary multiplication, improving on the algorithm in \cite{AFP22}, and this is the object of work in progress.

\textbf{Outline of the article.} In \S\ref{section: preliminaries}, we study several basic aspects of abelian threefolds with imaginary multiplication.
Among them, some already mentioned in this introduction: the field of definition of the endomorphisms, and the relation between the signature and the isogeny decomposition. In \S\ref{section: case i}, we prove part $i)$ of Theorem~\ref{theorem: maintheorem}.
We essentially apply a theorem of Casselman that attaches an elliptic curve to the algebraic Hecke character $\psi$ associated to the Hodge--Tate compatible system $\{W_{\lambda}(-1)\}_\lambda$ where $W_{\lambda}$ denotes $\wedge ^3V_{\lambda}(A)$.
Given part $i)$, part $ii)$ of Theorem~\ref{theorem: maintheorem} amounts to a descent result from $KM$ to $K$.
This is accomplished in \S\ref{section: case ii}. The conditions for this descent may be expressed in terms of the Hecke character $\psi$, and building on a result by Milne we establish the desired properties for $\psi$.
This is the key technical part of the article. In \S\ref{section: cohomological}, we give a reinterpretation of our main result in terms of the third étale cohomology group of an abelian threefold with imaginary multiplication. In \S\ref{section: Q},  we note how for $K=\Q$ the modular results by Hecke, Eichler, and Shimura allow one to circumvent the use of Casselman's theorem.
In there, we exhibit explicit examples of genus 3 curves defined over $\Q$ whose Jacobians respectively have potential imaginary multiplication by $\Q(\sqrt{-2})$ and $\Q(\sqrt{-7})$. These examples were kindly shared with us by Shiva Chidambaran and Drew Sutherland. We conclude the article by raising the question of which quadratic imaginary fields of class number $1$ may occur as geometric endomorphism algebras of threefolds over $\Q$.

\textbf{Notation and terminology.} Throughout this article, $K$ is a number field and~$\bK$ denotes a fixed algebraic closure.
For $X$ an abelian variety defined over $K$ and $L$ a field extension of $K$, we write $X_L$ to denote the base change of $X$ from $K$ to $L$, and $\End(X_L)$ to denote the ring of endomorphisms defined over $L$.
In particular, $\End(X_K) = \End(X)$ denotes the ring of endomorphisms defined over $K$.

For $F$ a number field, we let $\cO_F$ denote its ring of integers, and  $h_F$, its class number.
For $n\in \Z$, a rational prime $\ell$, and a $\Q_\ell[G_K]$-module $V$, we write $V(n)$ to denote the $n$th Tate twist $V\otimes \chi_\ell^{n}$, where $\chi_\ell$ denotes the $\ell$-adic cyclotomic character.  
We denote the induction of a representation from $G_L$ to $G_K$ by $\Ind^K_L$.
Viewing $\OO_M$ as a lattice in $\CC$ via some embedding $M \hookrightarrow \CC$, we let $j(\OO_M)$ denote its $j$-invariant.

Let $\mathcal{X,Y}$ be $\QQ$-algebras with multiplicative identity elements and $\iota \colon \mathcal{X} \hookrightarrow \mathcal{Y}$ an injective map.
We call $\iota$ an algebra embedding, if it respects multiplication and is $\QQ$-linear with respect to the $\QQ$-structure of $\mathcal{X}$ and $\mathcal{Y}$. 
We shall further call $\iota$ a unital algebra embedding if it sends the multiplicative identity of $\mathcal{X}$ to that of $\mathcal{Y}$.

\textbf{Acknowledgements.} Thanks to Shiva Chidambaran and Drew Sutherland for providing the examples with potential multiplication by $\Q(\sqrt{-2})$ and $\Q(\sqrt{-7})$ exhibited in \S\ref{section: Q}. The validity of several results in \S\ref{section: introduction}, as well as of the second part of Lemma \ref{lemma: induction module}, was surmised by Shiva Chidambaran from the inspection of numerical data.
We are indebted to him for sharing this knowledge with us. 
Thanks to Jeff Achter, Matt Broe, Xavier Guitart, Jef Laga, Martin Orr, Francesc Pedret, Ari Shnidman, Marco Streng, Charles Vial and the anonymous referee for helpful discussions and suggestions.
The authors were partially supported by grant PID2022-137605NB-I00.
Fité was additionally supported by the Ramón y Cajal fellowship  RYC-2019-027378-I, the María de Maeztu program CEX2020-001084-M, and the grant 2021 SGR 01468.

\section{Preliminaries}\label{section: preliminaries}

In this section we collect some background results on abelian varieties with potential multiplication by a number field.

\begin{lemma}\label{lemma: algebra embedding}
Let $X$ be an abelian variety defined over $K$ of dimension $g$ and $F$ a finite extension of $\QQ$.
Suppose there is an algebra embedding $\iota \colon F \hookrightarrow \End^0(X_\kbar)$. Then the following are equivalent:
\begin{enumerate}[i)]
    \item $\iota$ is unital.
    \item The action of $\iota(F)$ on the regular differentials of $X_\bK$ is equivalent to a direct sum of embeddings $F \hookrightarrow \bK$.
\end{enumerate}
If the above hold, then no element in $\iota(F \setminus\QQ)$ sends a regular differential of $X_\bK$ to a rational multiple of itself.
Moreover, if $[F\colon \Q]$ does not divide $g$, then $F$ is not totally real.
\end{lemma}

\begin{proof}
    Fix an embedding $\bar{K} \hookrightarrow \CC$.
    Recall that for any extension $L/K$ we have a homomorphism of $\QQ$-algebras $\End^0(X_L) \rightarrow \End(H^0(X_L, \Omega^1_{A_L}))$ \cite[\S 2.8]{Shi98}.
    The base change of this representation to $\CC$, i.e., $\End^0(X_\CC) \rightarrow \End(H^0(X_\CC, \Omega^1_{X_\CC}))$ is equivalent to the analytic representation $\rho_\CC \colon \End^0(X_\CC) \rightarrow \M_g(\CC)$ \cite[\S 3.2]{Shi98}.
    The rational representation $\rho_\QQ \colon \End^0(X_\CC) \rightarrow \M_{2g}(\QQ)$ is equivalent to $\rho_\CC \oplus \bar{\rho}_\CC$ \cite[\S 3.2]{Shi98}.

Suppose $\iota$ is a unital algebra embedding.
    Then \cite[\S5 Lemma 1]{Shi98} gives us that  
    $$
    \rho_\QQ|_{\iota(F)} \simeq \bigoplus_{\sigma}\sigma^{\oplus h}
    $$ 
    where $\sigma$ runs over the embeddings of $F$ into $\C$ and $h=2g/[F:\Q]$.    

    Conversely, suppose that the action of $\iota(F)$ on the regular differentials is given by a direct sum of embeddings $F \hookrightarrow \C$.
    Then the same is true for $\rho_\QQ$ restricted to $\iota(F)$.
    Let $\alpha$ be a primitive element of $F$ and let $\tau$ denote $\iota(\alpha)$.
    The subalgebra $\QQ(\tau)$ of $\End^0(X_\kbar)$ generated by $\tau$ is isomorphic to $\QQ[t]/\langle m_\tau \rangle$ where $m_\tau$ is the minimal polynomial of $\tau$.
    As $\rho_\QQ$ is a faithful representation of $\End^0(X_\kbar)$, the minimal polynomial of $\rho_\QQ(\tau)$ is equal to $m_\tau$.
    The restriction of $\rho_\QQ$ to $\iota(F)$ being a direct sum of embeddings implies the minimal polynomial of $\rho_\QQ(\tau)$ is equal to that of $\alpha$.
    Thus $\QQ(\tau) \simeq \QQ[t]/\langle m_\tau \rangle \simeq F$ and so $\iota(1) = \mathrm{Id}_{X_\kbar}$.

    In particular, the eigenvalues of an element in $\iota(F\setminus \Q)$ acting on the regular differentials are conjugates of this element, and so no element in $\iota(F \setminus\QQ)$ sends a regular differential of $X_\bK$ to a rational multiple of itself.

     Since $[F\colon \Q]$ does not divide $g$, we see that there exist $n_\sigma\in \Z$, not all equal, such that
    $$
    \rho_\CC |_{\iota(F)}\simeq \bigoplus_{\sigma}\sigma^{ \oplus n_\sigma}.
    $$
    On the other hand, as $\rho_\CC \oplus \bar{\rho}_\CC\simeq \rho_\QQ$, we find $n_\sigma+n_{\overline \sigma}=h$ for all $\sigma$. We deduce the existence of some $\sigma$ for which $\sigma\not=\overline \sigma$, and hence $F$ cannot be totally real.
  
\end{proof}

\begin{remark}
Let $X$ and $F$ be as in the previous lemma and suppose that $\iota(F)$ lies in the centre of $\End^0(X_\kbar)$. By \cite[\S5 Proposition 5]{Shi98}, then $F$ is either totally real or a CM field. By the previous lemma, if $[F\colon \Q]$ does not divide $g$, then $F$ must be a CM field.
\end{remark}
 
\begin{remark}\label{remark: product of iotas} 
Let $\varphi\colon X\rightarrow Y\times Z$ be an isogeny, where $X$, $Y$, and $Z$ are abelian varieties defined over $K$, with $\Hom(Y_\kbar,Z_\kbar)$ possibly nonzero. Let $F$ be a finite extension of $\QQ$.
Suppose there is a unital algebra embedding $\iota_X \colon F \hookrightarrow \End^0(X_\kbar)$ and algebra embeddings $\iota_Y \colon F \hookrightarrow \End^0(Y_\kbar)$ and $\iota_Z \colon F \hookrightarrow \End^0(Z_\kbar)$
making the diagram
$$
\xymatrix{
F \ar[r]^{\iota_X\hspace{5mm}} \ar[rd]_{\iota_Y \times \iota_Z} &   \End^0(X_\kbar) \\
& \End^0(Y_\kbar)\times \End^0(Z_\kbar)\ar@{^{(}->}[u]
}
$$
commutative, where the vertical arrow is the injection induced by $\varphi$. Then $\iota_Y$ and $\iota_Z$ are unital. 
\end{remark} 
 
\begin{example}
\label{example:hyperelliptic_embedding_not_an_algebra_embedding}
    Let $C$ denote the smooth projective curve determined by the smooth affine model $y^2 = x^8 +11x^4+3$.
    The order four automorphism $(x,y) \mapsto (ix,y)$ induces an automorphism on $C$ and $\Jac(C)$.
    We thus have an element $\tau$ of order 4 contained in $\End^0(\Jac(C)_\Qbar)$.

    

    The minimal polynomial of $\tau$ divides $t^4-1$ and hence the subalgebra $\QQ[\tau]$ generated by $\tau$ is isomorphic to a quotient of $\Q[t]/\langle t^4-1\rangle \simeq \Q(i) \times \Q \times \Q$ (with $\Q \cdot \mathrm{Id}_{\Jac(C)_\Qbar}$ sitting diagonally in such a product).
    Moreover, the $\Q(i)$ factor is not killed in this quotient, since $\tau$ has order 4.
    Thus we obtain an algebra embedding $\Q(i) \hookrightarrow \End^0(\Jac(C)_\Qbar)$.
    
    A basis of the regular differentials of $C$ is given by $\frac{dx}{y},\frac{xdx}{y},\frac{x^2dx}{y}$.
    Using this basis, one easily sees the eigenvalues of $\tau$ are $i,-1,-i$ respectively.
    The correspondence between the differentials of $C$ and $\Jac(C)$ allows us to see that any element in $\QQ[\tau]$ will have at least one rational eigenvalue.
    The above lemma applies to show that $\Q(i) \hookrightarrow \End^0(\Jac(C)_\Qbar)$ is not a unital algebra embedding.
\end{example}

\begin{example}
    Let $C$ denote the smooth projective curve determined by the affine smooth model $y^4 = x^3-x+1$.
    The order four automorphism $(x,y) \mapsto (x,iy)$ induces an automorphism $\tau$ on $C$ and $\Jac(C)$.
    As in the previous example, we obtain an algebra embedding $\Q(i) \hookrightarrow \End^0(\Jac(C)_\Qbar)$.
    
    A basis of the regular differentials of $C$ is given by $\frac{dx}{y^3},\frac{dx}{y^2},\frac{xdx}{y^3}$.
    Considering the action of $\tau$ on this basis, we deduce as in the above example that $\Q(i) \hookrightarrow \End^0(\Jac(C)_\Qbar)$ is not a unital algebra embedding.
\end{example}

From now on, we assume, without loss of generality, that $M$ is an imaginary quadratic field.

\begin{proposition}\label{proposition: prel1}
Let $A$ be an abelian threefold defined over $K$ with potential imaginary multiplication by $M$.
The following hold:
\begin{enumerate}[i)]
\item If $\iota(M)\subseteq  \End^0(A)$, then $M\subseteq K$.
    \item Suppose that $\End^0(A_\bK)$ is a field. If $M\subseteq K$, then $\iota(M)\subseteq  \End^0(A)$.
\end{enumerate}
\end{proposition}

\begin{proof}        
    Take $\alpha \in M\setminus \Q$.
    Suppose first that $\iota(M)\subseteq \End^0(A)$.
    We shall show $K$ contains $M$.
    By functoriality, $\iota(M)$ acts on $\wedge^3 H^0(A, \Omega^1_{A})$.
   By Lemma \ref{lemma: algebra embedding}, the action of $\alpha$ on this 1 dimensional $K$-vector space is given by multiplication by the product of 3 conjugates of $\alpha$, hence by an element of $M\setminus \Q$.
    In particular, if $\iota(M) \subseteq \End^0(A)$, then $K$ must contain $M$.

    Let us suppose now that $\End^0(A_\bK)$ is a field. 
    As $M \hookrightarrow\End^0(A_\bK)$, the field $\End^0(A_\bK)$ has either degree 2 or 6 \cite[Page 182, Corollary]{Mum}.
    The group $\GK$ acts on $\End^0(A_\bK)$ by $\QQ$-linear automorphisms.
    We claim that this action preserves $\iota(M)$.
    This is obvious if $\End^0(A_\bK)$ has degree 2 and it is implied by the discussion following \cite[Remark 4.5]{FKS25} if $\End^0(A_\bK)$ has degree 6.
    In particular, we obtain a homomorphism
    \[ \rho\colon \GK\rightarrow \Aut(\End^0(A_{\bK})/\Q) \rightarrow \Gal(\iota(M)/\QQ)\simeq \Gal(M/\Q),\]
    for which the central arrow is restriction. We will let $s\in G_K$ act on $M$ via the above homomorphism, so that the formula $\acc s\iota(\alpha) = \iota(\acc {\rho(s)} \alpha)$ holds.
    We want to show that if $M\subseteq K$, then $\rho$ is trivial.
    For this, it suffices to show that $\acc {\rho(s)} \alpha=\alpha$ for all $s\in G_K$.
        
    Given $\tau \in \End(A_{\bK})$, we let $P_\tau(T)$ denote the characteristic polynomial of $\tau$ acting on $H^0(A_{\bK}, \Omega^1_{A_{\bK}})$. Let $\alpha\in M\setminus \Q$. Since $M\subseteq K$, we have ${}^sP_{\iota(\alpha)}(T)=P_{\iota(\alpha)}(T)$. From the definition of $\rho$, we also have ${}^sP_{\iota(\alpha)}(T)=P_{\iota({}^{\rho(s)}\alpha)}(T)$. Hence $\rho(s)$ induces a permutation of order at most $2$ on the set of roots of $P_{\iota(\alpha)}(T)$. Since $P_{\iota(\alpha)}(T)$ has degree three, one of the roots must be fixed by the action of $\rho(s)$. This means that $\sigma({}^{\rho(s)}\alpha)=\sigma(\alpha)$ for some embedding $\sigma \colon M\hookrightarrow \bK$, which implies $\rho(s)$ is trivial. 
\end{proof}

The below proposition may be deduced from a result due to Shimura; see \cite{Shi63} or \cite[Exercise 9.10 (3)]{BH04}.
However, for the convenience of the reader, we shall give a direct elementary proof.

\begin{proposition}[Shimura]
\label{proposition: prel2}
Let $A$ be an abelian threefold defined over $K$ with potential imaginary multiplication by $M$.
Suppose $\End^0(A_\kbar) \not \simeq \M_3(M)$. 
Then $(A,\iota)$ has signature $(2,1)$. 

In particular, if $\End^0(A_{\kbar})$ is a quadratic field, then $\End^0(A_\kbar)$ is isomorphic to $M$ and $(A,\iota)$ has signature $(2,1)$.
\end{proposition}

\begin{proof}
Fix an embedding $\bK \hookrightarrow \CC$.
To simplify notation, we shall write $M$ in place of $\iota(M)$.
    We shall assume the signature of $A_\CC$ is $(3,0)$ and prove $\End^0(A_\kbar)  \simeq \M_3(M)$.
    As  the signature of $(A,\iota)$ is either $(3,0)$ or $(2,1)$, this amounts to proving the contrapositive of the given statement.
    
    We may write $A_\CC(\CC) = \CC^3/\Lambda$ for some lattice $\Lambda$ in $\CC^3$.
    Let $\rho_\CC$ denote the analytic representation of $\End(A_\CC)$.
    As the signature of $A_\CC$ is $(3,0)$, there is an embedding $\sigma \colon M \hookrightarrow \CC$ such that for any basis of $\CC^3$, $\alpha \in \End(A_\CC)$, and $z = (u,v,w) \in \CC^3/\Lambda$, we have
    \[\rho_\CC(\alpha)z = (\sigma(\alpha)u,\sigma(\alpha)v,\sigma(\alpha)w) = \sigma(\alpha)z.\]

    After possibly replacing $A_\CC$ by an isogenous abelian variety, we may assume $M \cap \End(A_\CC) = \OO_M$, where $\OO_M$ is the maximal order in $M$.    
    The lattice $\Lambda$ is a finitely generated torsion-free $\OO_M$-module under $\rho_\CC|_{\OO_M}$.
    By the structure theorem for such modules over Dedekind domains (see \cite[Proposition 10.2.1]{Bro}), we may write $\Lambda = \rho_\CC(I_1)x_1+\rho_\CC(I_2)x_2+\rho_\CC(I_3)x_3$ for some ideals $I_1,I_2,I_3$ of $\OO_M$ and $x_1,x_2,x_3 \in \CC^3$.
    As remarked above, we have $\rho_\CC(\alpha) x_i = \sigma(\alpha)x_i$.
    Since $\RR \otimes \sigma(I_i) \simeq \CC$, we see that \[\left(\RR \otimes \rho_\CC(I_i)x_i\right)/\rho_\CC(I_i)x_i \simeq \left(\RR \otimes \sigma(I_i)x_i\right)/\sigma(I_i)x_i \simeq \CC/\sigma(I_i).\]

    Define a $\CC$-linear map $T\colon \CC^3 \rightarrow \CC^3$ by $x_i \mapsto e_i$ (where $e_i$ is 1 in the $i$-th component and zero elsewhere).
    Since $\Lambda$ is a lattice in $\CC^3$, the vectors $x_1,x_2,x_3$ are linearly independent over $\CC$, and thus $T$ is well defined and invertible.
    It follows that $T$ induces a biholomorphic map
    \[\CC^3/\Lambda \rightarrow \CC/\sigma(I_1) \times \CC/\sigma(I_2) \times \CC/\sigma(I_3).\]
    The $\CC/\sigma(I_i)$ are isogenous elliptic curves with complex multiplication by $\OO_M$ (see for example \cite[Chapter II, \S 1]{Sil94}, or \cite[Chapter II, Theorem 2]{Shi98} and its corollary).
    Hence $\End^0(A_\CC) \simeq \M_3(M)$
\end{proof}

\begin{remark}
Note that the converse of the above proposition does not hold.
Indeed there are both central and non-central unital algebra embeddings of $M$ into $\M_3(M)$.
Thus for any abelian threefold $A$ with $ \End^0(A_\kbar) \simeq \M_3(M)$, we may equip $A$ with multiplication by $M$ of signature $(3,0)$ or $(2,1)$.
Together with the above proposition, this implies that for every abelian threefold $A$ with potential multiplication by $M$ there exists a unital algebra embedding $\iota\colon M\hookrightarrow \End^0(A_\kbar)$ such that $(A,\iota)$ has signature $(2,1)$.
In Theorem \ref{theorem: maintheorem}, we assume that $\iota(M)\subseteq \End^0(A_{KM})$.
We emphasise that different choices of $\iota$ may be defined over different base fields; see Examples \ref{example: FS} and \ref{example: restriction of scalars}.
\end{remark}

The following provides an example of a non-central unital algebra embedding occurring in `nature'.

\begin{example}\label{example: FS}
    Let $C$ denote the smooth projective curve defined over $\Q$ determined by $y^2 = x^7-5x$. Let $\tau \in \End^0(\Jac(C)_{\Qbar})$ be the endomorphism induced by the map $(x,y) \mapsto (-x,iy)$. Write $M=\Q(i)$ and consider the algebra embedding $\iota \colon M \hookrightarrow \End^0(\Jac(C)_{\Qbar})$ defined by $\iota(i)=\tau$.
    
    A basis of the regular differentials of $y^2 = x^7-5x$ is given by $\frac{dx}{y},\frac{xdx}{y},\frac{x^2dx}{y}$.
    Moreover, this is a basis of eigenvectors for $\tau$, having $i,-i,i$ as respective eigenvalues.
    By Lemma \ref{lemma: algebra embedding}, the algebra embedding $\iota$ is unital and of signature $(2,1)$.
    Also, by \cite[Lemma 5.6]{FS16}, we have $\End^0(\Jac(C)_\Qbar) \simeq \M_3(M)$.
    
    Note that $C$ is generic in the sense of \cite[Definition 5.7]{FS16}.
    Hence, from \cite[Figure 6, Figure 7]{FS16}, we see that the $\Q$ dimension of $\End^0(\Jac(C)_M)$ is 4.
    The proof of \cite[Lemma 5.6]{FS16} provides an explicit isogeny $\varphi$ over $\Q$ from $\Jac(C)$ to $Y\times Z$, where $Y$ is an elliptic curve and $Z$ is an abelian surface, both defined over $\Q$.
    It is straightforward to verify that $\iota$ induces via $\varphi$ unital algebra embeddings $\iota_Y\colon  M \hookrightarrow \End^0(Y_M)$ (resp. $\iota_Z \colon M \hookrightarrow \End^0(Z_M)$) yielding a commutative diagram as in Remark \ref{remark: product of iotas}.
    We deduce that 
\begin{equation}\label{equation: algebra dec example}
\End^0(\Jac(C)_M)\simeq \End^0(Y_M)\times \End^0(Z_M),
\end{equation}
where $M \simeq\End^0(Y_M)$ and $M \simeq\End^0(Z_M)$.
An explicit computation with $\varphi$ shows that $(Y,\iota_Y)$ has signature $(0,1)$ and $(Z,\iota_Z)$ has signature $(2,0)$.
Let $\iota'\colon M\hookrightarrow \End^0(\Jac(C)_M)$ be the unital algebra embedding corresponding to $\iota_Y(\overline \cdot)\times \iota_Z$ under the identification \eqref{equation: algebra dec example}.
Then $(\Jac(C),\iota')$ has signature $(3,0)$.

We conclude this example by noting, however, that there must exist a unital algebra embedding $\theta\colon M\hookrightarrow \End^0(\Jac(C)_\Qbar)$ of signature $(2,1)$ such that $\theta(M)\not\subseteq \End^0(\Jac(C)_M)$.
Indeed, one may easily see that the subalgebra generated by the images of all such embeddings contains $M\times M \times M$, which is incompatible with~\eqref{equation: algebra dec example}.
\end{example}

\begin{example}\label{example: restriction of scalars}
Let $M$ be an imaginary quadratic field of class number $3$.
Then there exists an elliptic curve $E$ defined over a cubic extension of $\QQ$ such that $E$ has potential complex multiplication by $M$.
Let $A$ be Weil's restriction of scalars of $E$ down to $\Q$. Then $\End^0(A_{\Qbar})\simeq \M_3(M)$ and thus there exists a unital algebra embedding $\iota \colon M \hookrightarrow \End^0(A_{\Qbar})$ of signature $(2,1)$. By Theorem \ref{theorem: maintheorem}, if $K/\Q$ is such that $\iota(M)\subseteq \End^0(A_{KM})$, then necessarily $[K\colon\Q]\geq 3$. 
\end{example}

\section{Proof of Theorem \ref{theorem: maintheorem}: case $i)$}\label{section: case i}

We start by recalling the $\lambda$-adic representations that appeared in \S\ref{section: introduction}, and some of their basic properties that will be used in this and the following section. Let $X$ be an abelian variety defined over $K$, and $F$ a number field for which there exists a unital algebra embedding
\begin{equation}\label{equation: embedding theta}
\theta\colon F \hookrightarrow \End^0(X).
\end{equation} 
 For a prime $\ell$, we also denote by $\theta$ the action of $F$ on $V_\ell(X)$ obtained by precomposing $\End^0(X) \rightarrow \End(V_\ell(X))$ with the unital algebra embedding of (\ref{equation: embedding theta}).
 For every prime $\lambda$ of $F$ over $\ell$, define 
$$
V_\lambda(X)\coloneqq V_\lambda(X,\theta)\coloneqq V_\ell(X)\otimes_{\theta(F)\otimes \Q_\ell,\sigma_\lambda}F_\lambda.
$$ 
Here, $\sigma_\lambda\colon F\hookrightarrow F_\lambda$ is the natural inclusion and the tensor product is taken with respect to the projection map $\theta(F)\otimes \Q_\ell\rightarrow F_\lambda$ which renders the diagram
$$
\xymatrix{
\theta(F)\otimes \Q_\ell \ar[r]&   F_\lambda\\
F \ar[u]_{\theta} \ar[ru]_{\sigma_\lambda} &
}
$$
commutative.
It is convenient to note that $V_\lambda(X,\theta)$ depends on $\theta$, as will be the case for the various Hecke characters that we will soon attach to the system $\{V_\lambda(X,\theta)\}_\lambda$. 

Let $\varrho_\ell\coloneqq \varrho_{X,\ell}$ and $\varrho_\lambda:=\varrho_{X,\lambda}$ respectively denote the representations of $G_K$ afforded by $V_\ell(X)$ and $V_\lambda(X)$.
Let $\p$ be a prime of $K$ of good reduction for $X$ and let $\Frob_\p$ denote a(n arithmetic) Frobenius element at $\p$.
By \cite[Theorem 2.1.2]{Rib76}, there exists $a_\p\in F$ independent of $\lambda$ and $\ell$ such that
$\Tr(\varrho_\lambda(\Frob_\p))=\sigma_\lambda(a_\p)$.
One says that $\{ V_\lambda(A)\}_\lambda$ is a compatible system of $F_\lambda$-adic representations. 
It is shown in the course of the proof of \cite[Theorem 2.1.1]{Rib76} that
\begin{equation}\label{equation: l adic in lambda adic}
V_\ell(X)\simeq \bigoplus_{\lambda \mid \ell}\Res_{F_\lambda/\Q_\ell}V_\lambda(X)
\end{equation}
as $\Q_\ell[G_K]$-modules. 

Assume now that $F=M$ is an imaginary quadratic field. We will denote by $\overline \cdot$ the the nontrivial automorphism of $M$. Recall that $\varphi_\lambda:M_{\overline \lambda} \rightarrow M_{ \lambda}$ denotes the $\Q_\ell$-isomorphism rendering the diagram
\begin{align*}
  \xymatrix{
M_{\overline \lambda}  \ar[r]^{\varphi_\lambda} & M_{ \lambda}\\
M \ar[r]^{\overline \cdot }\ar[u]^{\sigma_{\overline \lambda} } & M  \ar[u]_{\sigma_{ \lambda }}}
\end{align*}
commutative, where the vertical arrows are the natural inclusions. 
Writing $\overline V_{ \lambda}(X)$ to denote $V_{\overline \lambda}(X)$ viewed as an $M_\lambda[G_K]$-module via the isomorphism $\varphi_\lambda$, one finds
\begin{equation}\label{equation: conjugate traces}
\Tr(\Frob_\p |\overline V_\lambda (X))=\sigma_\lambda(\overline a_\p).
\end{equation}
From \eqref{equation: l adic in lambda adic}, we obtain
$$
\Tr(\varrho_\ell(\Frob_\p))= \sum_{\lambda\mid \ell} \Tr_{M_\lambda/\Q_\ell}\Tr(\varrho_\lambda(\Frob_\p))= \Tr_{M/\Q}(a_\p).
$$
By \cite[Satz 3]{Fal83}, the Tate module $V_\ell(X)$ is semisimple. Therefore, from the last two equalities and the Chebotaryov density theorem, we deduce an isomorphism
\begin{equation}\label{equation: lambda plus lambdabar}
V_\ell(X)\otimes_{\Q_\ell} M_\lambda \simeq V_\lambda(X) \oplus \overline V_{ \lambda}(X)
\end{equation}
of $M_\lambda[G_K]$-modules. 

Let $A$ be an abelian threefold defined over $K$ with potential multiplication by $M$ as in \S\ref{section: introduction}.
In this section, we assume that $\iota(M)\subseteq \End^0(A)$.
To simplify notation, we will write $W_\lambda$ (resp. $\overline W_\lambda$) in place of $\wedge^3V_\lambda(A)$ (resp. $\wedge^3\overline V_\lambda(A)$).
As explained in \cite[Proposition 14, \S7]{Fit24}, the compatible system of $M_\lambda[G_K]$-modules $\{W_\lambda\}_\lambda$ is associated with an algebraic Hecke character (a Hecke character of type $A_0$ in the sense of \cite{Wei56}) of $K$ with values in $M$ and infinity type $r_\sigma\cdot \sigma\circ \Nm_{K/M} +r_{\overline \sigma}\cdot \overline\sigma\circ \Nm_{K/M}$, where $(r_\sigma,r_{\overline \sigma})$ is the signature of $A$.

\begin{theorem}\label{theorem: M in k}
Let $A$ be an abelian threefold defined over $K$ with imaginary multiplication by $M$.
Suppose that the pair $(A,\iota)$ has signature $(2,1)$.
Then there exists an elliptic curve $E$ defined over $K$ with complex multiplication by $M$ such that, for every prime $\lambda$ of $M$ over a rational prime $\ell$, there is an isomorphism
$$
\big(\wedge ^3V_\lambda(A)\oplus \wedge ^3\overline V_{\lambda}(A) \big)(-1)\simeq V_\ell(E)\otimes_{\Q_\ell}M_\lambda
$$
of $M_\lambda[G_K]$-modules. 
In particular, $M$ has class number $h_M\leq [K:\Q]/2$.
\end{theorem}

\begin{proof}
By the hypothesis, the compatible system $\{W_\lambda\}_\lambda$ is associated with an algebraic Hecke character of $K$ with values in $M$ and infinity type $2\cdot \sigma\circ \Nm_{K/M} +\overline\sigma\circ \Nm_{K/M}$.
The compatible system $\{W_\lambda(-1)\}_\lambda$ is hence associated to an algebraic Hecke character $\psi$ of infinity type $\sigma \circ \Nm_{K/M}$.
By Proposition \ref{proposition: prel1}, we have $M\subseteq K$ and then, by Theorem \ref{theorem: Casselman} below, there exists an elliptic curve $E$ defined over $K$ with complex multiplication by $M$ such that the compatible system $\{ V_\lambda(E)\}_\lambda$ is associated to $\psi$.
By \eqref{equation: lambda plus lambdabar} applied to $X=E$, we have that $V_\ell(E)\otimes_{\Q_\ell}M_\lambda\simeq V_\lambda(E)\oplus \overline V_{\lambda}(E)$ as $M_\lambda[G_K]$-modules.
By \eqref{equation: conjugate traces} respectively applied to $X=A$ and $X=E$, we have that $\overline \psi$ is associated to both $\{\overline W_{\lambda}(-1)\}_\lambda$ and $\{ \overline V_{\lambda}(E)\}_\lambda$.
The isomorphism of the statement of the theorem follows.

By \cite[Theorem 1.1]{Rem17} there exists an elliptic curve $E'$ defined over $K$ and $K$-isogenous to $E$ such that $\End(E')\simeq \cO_M$. Let $j$ denote the $j$-invariant of $E'$, so that $\Q(j)\subseteq K$.
We claim that $\Q(j)\subsetneq K$.
Indeed, by \cite[Chapter II, Theorem 4.3]{Sil94}, we have $[\Q(j):\Q]=[M(j):M]=h_M$, which in particular implies that $M\not\subseteq \Q(j)$. Now recall that $M\subseteq K$.
It follows that
$$
h_M=[\Q(j):\Q]\leq [K:\Q]/2,
$$
as desired.
\end{proof}

Below we state a simplified version of a theorem of Casselman that is used in the proof of the above theorem. It associates an abelian variety with complex multiplication to an algebraic Hecke character satisfying certain properties. Let us first recall how to attach an algebraic Hecke character to an abelian variety with complex multiplication. Let $X$, $F$, and $\theta$ be as in the beginning of this section, so that \eqref{equation: embedding theta} is satisfied.
Assume that $[F:\Q]=2\dim(X)$, that is, that $X$ has complex multiplication.
In this case, the signature of $(X,\theta)$ is of the form $\sum_{\sigma \in \Phi}\sigma$, where $\Phi\subseteq \Hom(F,\bK)$ is a CM type of $F$, that is, it satisfies $\Phi\cup \overline \Phi=\Hom(F,\bK)$.
Moreover, after taking an embedding $\bK\hookrightarrow \C$, one has $X_\C\simeq \C^{\dim(X)}/\Phi(\mathfrak a)$, where $\mathfrak a\subseteq F$ is a lattice. 

Note that $V_\lambda(X,\theta)$ is a 1-dimensional $F_\lambda$-vector space \cite[Theorem 2.1.1]{Rib76}.
By \cite[Theorem 2.1.3]{Rib76}, it is of Hodge--Tate type and hence locally algebraic by \cite[Theorem on III-7]{Ser89}. By \cite[Theorem 2 on III-13]{Ser89}, the system $\{ V_\lambda(X,\theta)\}_\lambda$ is associated with an algebraic Hecke character.
We call this the algebraic Hecke character attached to $(X,\theta)$.
Assume further that $F/\Q$ is Galois and that $\Phi$ is primitive, that is, if $g\in \Gal(F/\Q)$ satisfies $g\Phi=\Phi$, then $g$ is trivial.
In this case, one has that the infinity type of this Hecke character is $\sum_{\sigma\in \Phi} \sigma^{-1} \circ\Nm_{K/F}$ (in general, the infinity type is the reflex type of $\Phi$, but under the Galois and primitivity hypothesis the reflex type takes the stated form; see \cite[\S 8.4 Examples (I)]{Shi98}).

\begin{theorem}\label{theorem: Casselman}
Let $F\subseteq K$ be a CM field which is Galois over $\Q$, $\Nf$ a nonzero integral ideal of $K$, and denote by $I_\Nf$ the group of fractional ideals of $K$ coprime to $\Nf$.
Let $\psi\colon I_{\Nf}\rightarrow F^\times$ denote an algebraic Hecke character of $K$, with values in $F$, and conductor $\Nf$.
Suppose that the infinity type of $\psi$ is of the form $\sum_{\sigma\in \Phi} \sigma^{-1} \circ \Nm_{K/F}$, where $\Phi\subseteq \Hom(F,\bK)$ is a primitive CM type of $F$.
Then there exists an abelian variety $X$ defined over $K$ and an algebra embedding
$$
\theta\colon F \hookrightarrow \End^0(X)
$$ 
such that $[F:\Q]=2\dim(X)$, $X$ has CM type $\Phi$, the algebraic Hecke character attached to $(X,\theta)$ is $\psi$, and $X_\C\simeq \C^{\dim(X)}/\Phi(\cO_F)$ after taking an embedding $\bK\hookrightarrow \C$.
Moreover, any pair $(X',\theta')$ also satisfying these conditions is $K$-isomorphic to $(X, \theta)$.
\end{theorem}

\begin{remark}
By saying that two pairs $(X,\theta)$ and $(X',\theta')$ as in the previous theorem are $K$-isomorphic, we mean that there exists an isomorphism $X\simeq X'$ defined over $K$ inducing an isomorphism $\varphi\colon \End^0(X)\simeq \End^0(X')$ that makes the diagram
$$
\xymatrix{
F \ar[r]^{\theta \hspace{5mm}} \ar[rd]_{\theta'} &   \End^0(X) \ar[d]^{\varphi}\\
& \End^0(X')
}
$$
commutative.
We also note that by \cite[Chapter 1, Theorem 3.5]{Lan83}, under the Galois and primitivity hypothesis of the theorem, $\theta$ is in fact an isomorphism.
\end{remark}

\begin{proof}
Let $\A_K^\times$ denote the idèles of $K$. As stated in \cite[Chapter 0, \S 5]{Sch88}, there exists a unique continuous group homomorphism 
$$
\psi_\A:\A_K^\times \rightarrow F^\times
$$
such that:
\begin{enumerate}[i)]
\item $\psi_\A(a)=\prod_{\sigma\in \Phi}\sigma^{-1}(\Nm_{K/F}(a))$ for all $a \in K^\times$.
\item $\psi_\A(s_\p)=\psi(\p)$ for every prime ideal $\p$ of $K$ not dividing $\Nf$. Here, $s_\p$ is an idèle whose $\p$-th component is a uniformiser of the completion of $K$ at $\p$ and whose all other components are $1$.
\end{enumerate}
To show the existence statement in the theorem, by taking $\mathfrak a=\cO_F$ in \cite[Chapter 5, Theorem 5.1]{Lan83}, it suffices to verify that for every $s\in \A_K^\times$, we have
\begin{equation}\label{equation: Casselmans conditions}
\psi_\A(s)\cdot \overline {\psi_\A(s)}=\Nm(s),\qquad \psi_\A(s)\cO_F=\prod_{\sigma\in \Phi}\sigma^{-1}(\Nm_{K/F}(s))\cO_F.
\end{equation}
By the above, it suffices to check \eqref{equation: Casselmans conditions} when $s$ is an element in $K^\times$ and when $s$ is the idèle $s_\p$ for $\p$ not dividing $\Nf$. In the first case, \eqref{equation: Casselmans conditions} follows from i) and the fact that $\Phi$ is a CM type. Indeed, for $a\in K^\times$ we have
$$
\psi_\A(a)\cdot \overline {\psi_\A(a)}=\prod_{\sigma\in \Phi\cup \overline \Phi}\sigma^{-1}(\Nm_{K/F}(a))=\Nm_{F/\Q}\Nm_{K/F}(a)= \Nm(a).
$$
When $s=s_\p$, then \eqref{equation: Casselmans conditions} can be recast as
\begin{equation}\label{equation: basic prop psi}
\psi(\p)\cdot \overline{\psi(\p)}=\Nm(\p),\qquad \psi(\p)\cO_F=\prod_{\sigma\in \Phi}\sigma^{-1}(\Nm_{K/F}(\p))\cO_F.
\end{equation}
We claim that it suffices to prove the above equalities after raising them to the $n$-th power for a given natural number $n$.
Indeed, as for the left-hand side equality, the claim follows from the fact that, after taking any embedding $F\hookrightarrow \C$, we have that $\psi(\p)\overline{ \psi(\p)}$ is a positive real number.
As for the right-hand side equality, the claim follows from unique factorisation of fractional ideals.
By finiteness of the class number, there exists a natural number $n$ and $a\in K^\times$ such that $\p^n=(a)$.
This reduces the case $s_\p$ to the already proven case $a\in K^*$.

As for the unicity statement, observe that a priori this does not  directly follow from the unicity statement in \cite[Chapter 5, Theorem 5.1]{Lan83}, as the latter requires singling out a polarisation. The proof, however, based on \cite[Chapter 5, Theorem 5.2]{Lan83}, shows that unicity holds also after forgetting the polarisation. 
\end{proof}

\section{Proof of Theorem \ref{theorem: maintheorem}: case $ii)$}\label{section: case ii}

Throughout this section we make the following assumption on the abelian threefold $A$ over $K$ with potential multiplication by $M$.
 
\begin{assumption}\label{assumption: quadratic field condition}
There exists a quadratic extension $L/K$ such that 
$$
[\iota(M)\cap \End^0(A_L):\iota(M)\cap \End^0(A)]=2.
$$
In other words, $\iota(M)\subseteq \End^0(A_L)$ and $\iota(M)\cap \End^0(A)= \Q$.
\end{assumption}
Observe that, by Proposition \ref{proposition: prel1}, if Assumption \ref{assumption: quadratic field condition} holds, then necessarily $M\subseteq L$. Let 
$$
\varrho_{\ell}:=\varrho_{A,\ell}:G_{K}\rightarrow \Aut(V_\ell(A)),\qquad \varrho_{\lambda}:=\varrho_{A_L,\lambda}:G_{L}\rightarrow \Aut(V_\lambda(A_L))
$$ 
respectively denote the $\ell$-adic and $\lambda$-adic representations attached to $A$ and $A_L$.
Let $s \in G_K$ be an element which does not map to the identity in $\Gal(L/K)$ and write ${}^s\varrho_{\lambda}$ to denote the representation defined by $\acc{s}\varrho_{\lambda}(t)\coloneqq\varrho_{\lambda}(s t s^{-1})$ for $t\in G_{L}$.
Let $\acc{s}V_{\lambda}(A_L)$ be the $M_\lambda[G_L]$-module affording the representation $\acc s \varrho_\lambda$.
Define similarly $\acc{s}V_{\ell}(A)$ and observe that $V_\ell(A)\simeq \acc{s}V_{\ell}(A)$ as $\Q_\ell[G_K]$-modules.
From Assumption \ref{assumption: quadratic field condition} one has an isomorphism $\Gal(L/K)\simeq \Gal(M/\Q)$, under which we will let $s$ act on $M$ via its nontrivial automorphism. Hence we may write $\acc s\lambda=\overline \lambda$. 

To shorten notation, we write $W_\lambda$ (resp. $\overline W_\lambda$) in place of $\wedge^3 V_\lambda(A_{L})$ (resp. $\wedge^3 \overline V_\lambda(A_{L})$).
If the pair $(A,\iota)$ has signature $(2,1)$, let $\psi$ denote the algebraic Hecke character of $L$ associated to $\{W_\lambda(-1)\}_\lambda$ as in the previous section.
Similarly, write $\acc s\psi$ to denote the algebraic Hecke character associated to $\{\acc s W_\lambda(-1)\}_\lambda$.
It satisfies $\acc s\psi(\af)=\psi(\acc s\af)$, where $\af$ is a fractional ideal of $L$ coprime to the conductor of $\psi$.
By the results of the previous section, there is an elliptic curve $E_0$ defined over $L$ attached to $\psi$.
Lemmas \ref{lemma: complex comm} and \ref{lemma: psi on inert} below exhibit certain properties of $\psi$ that will enable us to obtain a model of $E_0$ over $K$.

In the course of the proof of \cite[Theorem 2.9]{FG22} it is essentially shown that $\acc sV_{\lambda}(A_L)\simeq \overline V_{ \lambda}(A_L)$.
Since some formal aspects are distinct in our context, we detail the proof of this isomorphism here for the reader's convenience.

\begin{lemma}\label{lemma: complex comm}
Let $A$ be an abelian threefold over $K$ satisfying Assumption \ref{assumption: quadratic field condition}. Then there is an isomorphism $\acc{s}V_{\lambda}(A_L)\simeq \overline V_{\lambda}(A_L)$ of $M_\lambda[G_L]$-modules. In particular, if the pair $(A,\iota)$ has signature $(2,1)$, then $\acc s\psi=\overline \psi$.
\end{lemma}

\begin{proof}
Associated to the abelian variety $\acc s A_L$ and the unital algebra embedding
$$
s(M) \rightarrow \End^0 (\acc s A_L) \rightarrow  \End_{G_L}(V_\ell(\acc s A_L)),
$$
we have the $M_{\acc s \lambda}[G_L]$-module
$$
V_{\acc s \lambda}(\acc s A_L):= V_\ell(\acc s A_L)\otimes_{s(M)\otimes \Q_\ell, \sigma_{\acc s\lambda}}M_{\acc s \lambda}.
$$
Since $A$ is defined over $K$, by transport of structure, there is an isomorphism $V_{\lambda}(A_L)\simeq V_{\acc s \lambda}(\acc s A_L)$.
The map $A_L(\bar{K}) \rightarrow \acc s A_L(\bar{K})$ given by $P \mapsto \acc s P$ induces an isomorphism of $M_\lambda[G_L]$-modules $\acc {s^{-1}} V_{\acc s \lambda}(A_L) \simeq V_{\acc s \lambda}(\acc s A_L)$. The previous isomorphisms, together with the fact that $\overline V_\lambda(A_L)$ is  $V_{\acc s \lambda}(A_L)$ regarded as an $M_\lambda[G_L]$-module, imply the first part of the lemma.
Taking determinants, we obtain
$$
\acc{s} W_\lambda \simeq \overline W_{ \lambda},
$$
and the second part of the lemma follows.
We give an alternative more direct proof of the latter part of the lemma.
Indeed, on the one hand, we have that
$$
(A_L,\acc s \iota)\simeq (\acc s A_L,\acc s \iota),
$$
since $A_L$ is in fact defined over $K$. On the other hand, observe that the algebraic Hecke characters attached to the Hodge--Tate modules $\{W_\lambda(A_L,\acc s \iota)(-1)\}_\lambda$ and $\{W_\lambda(\acc s A_L,\acc s \iota)(-1)\}_\lambda$ are $\overline\psi$ and $\acc s\psi$, respectively. We conclude that $\overline\psi$ and $\acc s\psi$ must coincide.
\end{proof}

Recall that we denote the induction of a representation from $G_L$ to $G_K$ by $\Ind^K_L$.

\begin{lemma}\label{lemma: induction module} Let $A$ be an abelian threefold over $K$ satisfying Assumption \ref{assumption: quadratic field condition}.
Then, for every prime $\lambda$ of $M$ over a rational prime $\ell$, there is an isomorphism
$$
V_{\ell}(A)\otimes_{\Q_\ell} M_\lambda\simeq \Ind^{K}_L V_\lambda(A_{L})
$$
of $M_\lambda[G_K]$-modules.
In particular, if $\mathfrak p$ is a prime ideal of $K$ of good reduction for $A$ and inert in $L/K$.
Then there exists $t_\p\in \Z$ with $|t_\p|\leq 2\mathrm{Nm}(\p)$ such that 
$$
L_\p(A,T)=(1+\Nm(\p)T^2)(1-t_\p T^2 + \Nm(\p)^2 T^4),
$$
where $L_\p(A,T):=\det(1-\varrho_\ell(\Frob_\p)T)$.
\end{lemma}

\begin{proof}
Under Assumption \ref{assumption: quadratic field condition}, we can apply \cite[Theorem 3]{Mil72}. This yields an isogeny
$$
 A^2\sim\Res_{L/K}A_L ,
$$
where $\Res_{L/K}A_L$ denotes Weil's restriction of scalars. By \cite[Page 178, Paragraph (a)]{Mil72}, in terms of $M_\lambda[G_K]$-modules, this is recast as
$$
\big(V_\ell(A)\otimes_{\Q_\ell}M_\lambda\big) ^{\oplus 2} \simeq \Ind^K_L\left(V_\ell(A_L)\otimes_{\Q_\ell} M_\lambda\right).
$$
By successively using \eqref{equation: lambda plus lambdabar} and Lemma \ref{lemma: complex comm}, we obtain
$$
V_\ell(A_L)\otimes_{\Q_\ell} M_\lambda\simeq V_\lambda(A_L)\oplus \overline V_{\lambda}(A_L) \simeq V_\lambda(A_L)\oplus \acc s V_\lambda(A_L).
$$
Putting together the two previous isomorphisms yields
$$
\big(V_\ell(A)\otimes_{\Q_\ell}M_\lambda\big) ^{\oplus 2}  \simeq \Ind^K_LV_\lambda(A_L)\oplus \Ind^K_L\acc s V_\lambda(A_L)\simeq \big(\Ind^K_L V_\lambda(A_L)\big)^{\oplus 2},
$$
which establishes the first part of the lemma.

It follows\footnote{Here, we use the following fact.
Let $G$ be a group, $H$ a normal subgroup of finite index of $G$, and $\varrho$ a finite dimensional representation of $H$.
Then, for $s \in G$, the characteristic polynomial of $\varrho(s)$ is a polynomial in $T^f$, where $f$ is the order of $s$ in $G/H$.
This is shown, for example, in the course of the proof of \cite[Proposition 10.4 (iv)]{Neu92}.} from the above isomorphism that, if $\p$ is inert in $L/K$, then
$L_\p(A,T)$ is a polynomial in $T^2$, that is, there exists $b_\p\in \Z$ such that
$$
L_\p(A,T)=1+b_\p T^2 + b_\p \Nm(\p) T^4 + \Nm(\p)^3 T^6.
$$
Such a polynomial decomposes as 
$$
L_\p(A,T)=(1+\Nm(\p)T^2)(1+(b_\p-\Nm(\p)) T^2 + \Nm(\p)^2 T^4).
$$
It thus suffices to show that $-\Nm(\p)\leq b_\p \leq 3\Nm(\p)$. But this is precisely the bound for $b_\p$ given by \cite[Proposition 4]{KS08}. 
\end{proof}

\begin{lemma}\label{lemma: psi on inert}
Let $A$ be an abelian threefold over $K$ satisfying Assumption \ref{assumption: quadratic field condition} such that the pair $(A,\iota)$ has signature $(2,1)$.
Let $\mathfrak p$ be a prime ideal of $K$ of good reduction for $A$ inert in $L/K$.
Letting $\PP$ denote the prime of $L$ above $\p$, we have $\psi(\PP)=-\Nm(\p)$.
\end{lemma}

\begin{proof}
Equation \eqref{equation: basic prop psi} and Lemma \ref{lemma: complex comm} imply
$$
\psi(\PP)^2=\psi(\PP)\psi(\acc s\PP)=\psi(\PP)\overline\psi(\PP)=\Nm(\PP)=\Nm(\p)^2.
$$
Hence we have $\psi(\PP)=\pm \Nm(\p)$, and the question is to establish the minus sign.
The roots of $L_\PP(A_{L},T)$ are the squares of the roots of $L_\p(A,T)$.
Lemma~\ref{lemma: induction module} then implies that the reciprocal roots of $L_\PP(A_{L},T)$ are of the form $-\Nm(\p)$, $-\Nm(\p)$, $\alpha$, $\alpha$, $\overline \alpha$, $\overline \alpha$ for some $\alpha \in \Qbar$ such that $\alpha\overline \alpha=\Nm(\p)^2=\Nm(\PP)$.
From \eqref{equation: lambda plus lambdabar}, we deduce the eigenvalues of $\varrho_\lambda(\Frob_\PP)$ are either of the form
\begin{equation}\label{equation: two possible eigenvalues}
-\Nm(\p),\alpha,\alpha   \qquad \text{or} \qquad -\Nm(\p),\alpha,\overline\alpha.
\end{equation} 
Recall that there is $a_\PP\in M$ such that $\Tr(\varrho_\lambda(\Frob_\PP))=\sigma_\lambda(a_\PP)$. We claim that in fact $a_\PP$ belongs to $\Q$.
Indeed, since $\acc s\PP=\PP$, we have
$$
\sigma_\lambda(a_\PP)=\Tr(\varrho_\lambda(\Frob_{\acc s\PP}))=\Tr(\varrho_\lambda(s\Frob_{\PP}s^{-1}))=\Tr(\acc s\varrho_{\lambda}(\Frob_\PP))=\Tr(\Frob_\PP|\overline V_\lambda(A))=\sigma_\lambda(\overline a_\PP),
$$
where we have used Lemma \ref{lemma: complex comm} in the penultimate equality and \eqref{equation: conjugate traces} in the last equality.
Hence the eigenvalues of $\varrho_\lambda(\Frob_\PP)$ are given by the latter of the two options of \eqref{equation: two possible eigenvalues}, and thus we have
$$
\psi(\PP)=\frac{\det(\varrho_\lambda(\Frob_\PP))}{\Nm(\PP)}=-\frac{\Nm(\p)\alpha \overline \alpha}{\Nm(\PP)}=-\Nm(\p),
$$
as desired.
\end{proof}

\begin{theorem}\label{theorem: M not in k}
Let $A$ be an abelian threefold defined over $K$ with potential imaginary multiplication by $M$ via the unital algebra embedding $\iota \colon M\hookrightarrow \End^0(A_\bK)$.
Let $L$ denote $KM$ and suppose that 
$$
[\iota(M)\cap \End^0(A_L)\colon \iota(M)\cap \End^0(A)]=2.
$$ 
Suppose that the pair $(A,\iota)$ has signature $(2,1)$ and that there exists an embedding $\sigma\colon L \hookrightarrow \C$ such that either $\sigma(K)\subseteq \R$ or $j(\cO_M)\in \sigma(K)$. 

Then there exists an elliptic curve $E$ defined over $K$ such that $E_L$ has complex multiplication by $M$ and  for every prime $\lambda$ of $M$ over a rational prime $\ell$, there is an isomorphism
$$
\Ind^K_L\big(\wedge ^3V_\lambda(A_L) \big)(-1)\simeq V_\ell(E)\otimes_{\Q_\ell}M_\lambda
$$
of $M_\lambda[G_K]$-modules.
In particular, $M$ has class number $h_M\leq [K\colon \Q]$.
\end{theorem}

\begin{proof}
By Theorem \ref{theorem: M in k} applied to $A_L$, there exists an elliptic curve $E_0$ defined over $L$ for which there is an isomorphism 
$$
\big(W_\lambda\oplus \overline W_{\lambda}\big)(-1)\simeq V_\ell(E_0)\otimes_{\Q_\ell}M_\lambda
$$
of $M_\lambda[G_{L}]$-modules. Moreover, we know that $E_0$ has complex multiplication by $M$, say, via the unital algebra embedding
$$
\iota_{E_0}\colon M\hookrightarrow \End^0(E_0).
$$ 
In order to simplify the notation, let us write $\xi_\lambda$ (resp. $\overline \xi_\lambda$) to denote the representation of $G_L$ afforded by $W_\lambda(-1)$ (resp. $\overline{W}_\lambda(-1)$) until the end of this proof.
We claim that to prove the theorem it suffices to show that there exists an elliptic curve $E$ defined over $K$ such that $E_{L}$ and $E_0$ are isomorphic (over $L$).
Indeed, in that case, for $t\in G_L$, we have
$$
\Tr\Ind^K_L\xi_\lambda(t)=\Tr\xi_\lambda(t)+\Tr\overline \xi_{\lambda}(t)=\Tr \varrho_{E_0,\ell}(t)=\Tr \varrho_{E,\ell}(t).
$$ 
On the other hand, for $t\in G_K\setminus G_L$, we have
$$
\Tr\Ind^K_L \xi_\lambda(t)=0=\Tr \varrho_{E,\ell}(t).
$$
This implies that $\Ind^K_L \xi_\lambda\simeq \varrho_{E,\ell}$, since a semisimple representation with coefficients in a field of characteristic 0 is determined up to isomorphism by its trace (see \cite[Lemma on I-11]{Ser89}).

We denote the base change $E_0\times _{L,\sigma} \C$ by $E_{0,\C}$. 
To show that $E_0$ admits a model over $K$, it suffices to show that conditions (1), (2), and (3'') of \cite[Chapter 5, Theorem 6.3]{Lan83} hold. In our setting, these are:
\begin{itemize}
\item [(1)] $\acc s \psi=\overline \psi$.
\item [(2)] $\sigma(K)$ contains the field of moduli of $E_{0,\C}$.
\item [(3'')] There exists a prime $\mathfrak p$ of $K$ inert in $L/K$ such that, if $\PP$ denotes the prime of $L$ above $\p$, then $\psi(\PP)=-\Nm(\p)$.
\end{itemize}
 Conditions (1) and (3'') are satisfied by Lemmas \ref{lemma: complex comm} and \ref{lemma: psi on inert}, respectively.
By Theorem \ref{theorem: Casselman}, we may choose $E_0$ so that $E_{0,\C}\simeq \C/\cO_M$. This implies that the field of moduli of $E_{0,\C}$ is $\Q(j(\cO_M))$. Hence (2) holds if $j(\cO_M)$ lies in $\sigma(K)$. 

It suffices to show that (2) holds if $\sigma(K)\subseteq \R$. In this case, $s$ corresponds to complex conjugation on $\sigma(L)$, and hence $\sigma(K)$ is the subfield of $\sigma(L)$ fixed by complex conjugation. We have that 
$$
 \overline E_{0,\C}\simeq \C/ \overline\cO_M\simeq \C/\cO_M\simeq E_{0,\C}.
$$
Thus $\sigma(K)$ contains the field of moduli of $E_{0,\C}$. 
\end{proof}

\section{Cohomological interpretation}\label{section: cohomological}

In this section we will give an interpretation of Theorem \ref{theorem: maintheorem} in terms of étale cohomology. Let $A$ be an abelian threefold defined over $K$. For every prime $\ell$, the Weil pairing and the choice of a polarisation yield a nondegenerate alternating $G_K$-equivariant bilinear form
$$
V_\ell(A) \times V_\ell (A) \rightarrow \Q_\ell(1).
$$
This defines a nontrivial element $w$ in $\big( \big(\wedge ^2 V_\ell(A)\big) (-1)\big) ^{G_K}$. An easy calculation shows that the map 
$$
V_\ell(A)(1) \xrightarrow{- \wedge w} \wedge^3 V_\ell(A),
$$
is injective. Together with the isomorphisms \cite[Theorem 15.1]{Mil86}
$$
H^1_\et (A_\kbar,\Q_\ell)\simeq  V_\ell(A)^\vee, \qquad H^3_\et (A_\kbar,\Q_\ell)\simeq \wedge^3 V_\ell(A)^\vee, 
$$
the above injection exhibits $H^1_\et(A_\kbar, \Q_\ell)$ as a $G_K$-subrepresentation of $H^3_\et(A_\kbar, \Q_\ell(1))$. One may ask what is the largest subrepresentation of $H^3_\et(A_\kbar, \Q_\ell(1))$ which is isomorphic to $H^1_\et(B_\kbar, \Q_\ell)$ for some abelian variety $B$ defined over $K$ (as Matt Broe informed us, more general variants of this question have been studied in \cite{Kah21} and \cite{ACV20}, and the references therein). In the case that $\End(A_\kbar)$ is just $\Z$, the next proposition tells us that the largest such subrepresentation is already given by $H^1_\et(A_\kbar, \Q_\ell)$.

\begin{proposition}
Let $A$ be an abelian threefold defined over $K$ such that $\End(A_\kbar)\simeq \Z$. If $H^1_\et(B_\kbar, \Q_\ell)$ is isomorphic to a $G_K$-subrepresentation of $H^3_\et(A_\kbar, \Q_\ell(1))$ for some nontrivial abelian variety $B$ defined over $K$, then $B$ is isogenous to $A$.
\end{proposition}

\begin{proof}
Let $U_\ell$ denote the complement of $V_\ell(A)(1)$ in $\wedge^3 V_\ell(A)$. We claim that $U_\ell$ is an irreducible representation of $G_K$.

Let $\Std$ denote the representation of the algebraic group $\GSp_{6}$ over $\Q_\ell$ afforded by the $\Q_\ell$-vector space $V_\ell(A)$, and let $\chi$ denote its similitude character. By the unitary trick \cite[\S5.3]{Ser12} and \cite[Page 258]{FH91}, there is a decomposition
\begin{equation}\label{equation: wedge 3 dec}
\wedge^3 \Std \simeq \Std \otimes \chi \oplus \Gamma,
\end{equation}
where $\Gamma: \GSp_6 \rightarrow \GL_{14}$ is an irreducible representation. It follows that the representation $\Gamma$ is afforded by $U_\ell$. Recall that we denote by $\varrho_\ell$ the $\ell$-adic representation obtained from the action of $G_K$ on $V_\ell(A)$. We will denote by 
$$
\Gamma\varrho_\ell:G_K \rightarrow \GSp_6(\Q_\ell) \rightarrow \GL_{14}(\Q_\ell)\simeq \Aut(U_\ell)  
$$ 
the representation obtained as the composition of $\varrho_\ell$ with the map on $\Q_\ell$-points induced by $\Gamma$. For a contradiction, suppose that there exists a proper nontrivial subspace $Y_\ell \subseteq U_\ell$ that is stable under the action of $\Gamma\varrho_\ell(G_K)$. Let $H \subseteq \GSp_6$ denote the stabiliser of the subspace $Y_\ell$ under the action via~$\Gamma$. Since $\Gamma$ is irreducible, $H$ is a proper algebraic subgroup of $\GSp_6$. Since $Y_\ell$ is stable under the action of $\Gamma\varrho_\ell(G_K)$, we deduce that $\varrho_\ell(G_K)$ is contained in $H(\Q_\ell)$. Hence, the Zariski closure of $\varrho_\ell(G_K)$ is contained in $H\subsetneq \GSp_6$. This contradicts Serre's open image theorem \cite[Page 51, Corollaire]{Ser13}.  

Hence, to prove the proposition it will suffice to show that $U_\ell$ cannot be isomorphic to $V_\ell(X)(1)$ for any abelian variety $X$ defined over $K$. 

For $\p$ a prime of $K$, let $p$ denote its residue characteristic, and write $v_p$ for an extension to $\Qbar$ of the $p$-adic valuation. By \cite[Proposition 8 and Proposition 10]{Fit24} (applied with $\theta$ equal to $\varrho_\ell$) there exists a positive density set $R$ of primes $\p$ of $K$ of good reduction for $A$ with the following property: if $\alpha_1,\dots ,\alpha_6$ are the eigenvalues of $\varrho_\ell(\Frob_\p)$ ordered so that $v_p(\alpha_i)\leq v_p(\alpha_j)$ if $i< j$, then $v_p(\alpha_1)=v_p(\alpha_2)=0$ and $v_p(\alpha_3)\leq 1/2$.

For a contradiction, suppose that $U_\ell$ is isomorphic to $V_\ell(X)(1)$ for some abelian variety $X$ defined over $K$. 
Let $\p$ be a prime in $R$ not lying above $\ell$. In particular, $\p$ is a prime of good reduction for $X$. On the one hand, any eigenvalue $\alpha$ of $\Frob_\p$ acting on $V_\ell(X)(1)$ must satisfy $v_p(\alpha)\geq 1$. On the other hand, since $\p$ lies in $R$, there is an eigenvalue of $\Frob_\p$ acting on $U_\ell$ of the form $\alpha_1 \alpha_2 \alpha_3$, where $\alpha_1,\alpha_2,\alpha_3$ are eigenvalues of $\varrho_\ell(\Frob_\p)$ with  $v_p(\alpha_1)=v_p(\alpha_2)=0$ and $v_p(\alpha_3)\leq 1/2$. This contradiction completes the proof of the proposition.
\end{proof}

The previous proposition should be compared with the following consequence of Theorem \ref{theorem: maintheorem}.

\begin{corollary}
Let $A$ be an abelian threefold defined over $K$ such that $\End^0(A_\kbar)$ is a field containing an imaginary quadratic field $M$.
If $KM \neq K$, then assume further that there exists a field embedding $\sigma\colon K \hookrightarrow \C$ such that either $\sigma(K)\subseteq \R$ or $j(\cO_M)\in \sigma(K)$.

Then there exists an elliptic curve $E$ defined over $K$ with potential complex multiplication by $M$ such that $H^1_\et(E_\kbar\times A_\kbar, \Q_\ell)$ is a $G_K$-subrepresentation of $H^3_\et(A_\kbar, \Q_\ell(1))$.
\end{corollary}

\begin{proof} 
Since $A$ is geometrically simple, we have
$$
\Hom_{G_K}(H^1_\et(E_\kbar, \Q_\ell),H^1_\et(A_\kbar, \Q_\ell))=0.
$$
In view of the first paragraph of this section, it will suffice to show that $H^1_\et(E_\kbar, \Q_\ell)$ is a $G_K$-subrepresentation of $H^3_\et(A_\kbar, \Q_\ell(1))$.
By Proposition \ref{proposition: prel1}, the minimal extension $L/K$ such that $\iota(M)\subseteq \End^0(A_L)$ is $L=KM$. By Theorem \ref{theorem: maintheorem} there exists an elliptic curve $E$ defined over $K$ with potential complex multiplication by $M$ such that
$$
  V_\ell(E_L)(1)\otimes _{\Q_\ell} M_\lambda \simeq \wedge^3 V_\lambda (A_L) \oplus \wedge^3 \overline V_\lambda(A_L) \subseteq\wedge^3_{M_\lambda} (V_\ell(A_L) \otimes_{\Q_\ell} M_\lambda)\simeq (\wedge^3_{\Q_\ell} V_\ell(A_L) )\otimes_{\Q_\ell} M_\lambda
$$
of $M_\lambda[G_L]$-modules. Dualising, we obtain $H^1_\et(E_\kbar, \Q_\ell)$ as a $G_L$-subrepresentation of $H^3_\et (A_\kbar,\Q_\ell(1))$. If $L=K$, this yields the corollary. Otherwise, we have that $[L:K]=2$, and, as in the previous case, we see that $V_\lambda(E_L)(1)^\vee$ is a $G_L$-subrepresentation of $H^3_\et (A_\kbar,\Q_\ell(1))\otimes_{\Q_\ell} M_\lambda$. Since $V_\ell(E)\otimes_{\Q_\ell} M_\lambda \simeq \Ind^K_L V_\lambda(E_L)$ as $G_K$ representations,  Frobenius reciprocity shows that $H^1_\et(E_\kbar, \Q_\ell)$ is a $G_K$-subrepresentation of $H^3_\et(A_\kbar, \Q_\ell(1))$, as desired. 
\end{proof}

\section{Field of definition $K=\Q$}\label{section: Q}

We conclude the article by making some considerations specific to the case in which the field of definition is $K=\Q$. 

\subsection*{Approach via modular forms.} We first note that, in this case, one can prove Theorem~\ref{theorem: maintheorem} by invoking the Eichler-Shimura construction instead of Theorem \ref{theorem: Casselman}.
Indeed, observe that  $M$ is imaginary by Lemma \ref{lemma: algebra embedding}.
Let $W_\lambda $ denote $\wedge^3V_\lambda(A_M)$.
Then, the Hecke character $\psi$ associated to $\{W_\lambda(-1)\}_\lambda$, of $M$ with values in $M$, corresponds, by Proposition \ref{proposition: prel2} and \cite[Lemma 3]{Shi71} or \cite[p. 717]{Hec59} (see also \cite[Theorem 3.4]{Rib77}), to a normalised weight 2 newform $f=\sum_{n\geq 1}c_nq^n$ with complex multiplication by $M$.
For all but finitely many primes $p$, the Fourier coefficients $c_p$ satisfy $c_p=0$ if $p$ is inert in $M$ and $c_p=\psi(\p)+\psi(\overline\p)$ if $p=\p \overline \p$ is split in $M$.
In other words, for all but finitely many primes $p$ we have that 
$$
c_p=\Tr(\Frob_p|\Ind^\QQ_M(W_\lambda (-1)).
$$
By Lemma \ref{lemma: complex comm}, if $p=\p \overline \p$ is split in $M$, we must have $c_p=\psi(\p)+\overline{\psi(\p)}$, which implies that $f$ has rational coefficients. 
Then, the abelian variety attached to $f$ by the Eichler--Shimura construction is an elliptic curve $E$ defined over $\Q$, it satisfies that $E_M$ has complex multiplication by $M$, and for all but finitely many primes $p$ we have that 
$$
c_p=\Tr(\varrho_{E,\ell}(\Frob_p)).
$$
Theorem \ref{theorem: maintheorem} then follows from the Chebotaryov density theorem and \cite[Chapter II, Theorem 4.3]{Sil94}.

\subsection*{Explicit examples.} Let $A$ be an abelian threefold defined over $\Q$ such that $M=\End^0(A_\Qbar)$ is an imaginary quadratic field. It follows from Corollary \ref{corollary: bounded class number} that $M$ is limited to $9$ possibilities. It is then natural to ask the following.

\begin{question}
Which imaginary quadratic fields occur as $\End^0(A_\Qbar)$ for some abelian threefold $A$ defined over $\Q$?
\end{question}

Table \ref{table: examples} shows examples realising $M=\Q(\sqrt {-d})$ for $d=1,2,3,7$.
The example with $d=3$ may be verified using \cite{Zarhin_superelliptic} and the rest using \cite{Rigorous_comp}. Examples with $d=1,3$ are easy to construct. The more exotic examples with $d=2,7$ were found by Shiva Chidambaran by examining the databases of genus 3 curves created by Drew Sutherland with Sato--Tate group $\Unitary(3)$, i.e., the Sato--Tate group compatible with the sought for geometric endomorphism algebra (see \cite{Sut19} and \cite[\S7.3]{FKS25}).

\begin{table}[bth!]
\setlength{\tabcolsep}{10pt}
\begin{tabular}{lr}
$d$ & Curve \\\midrule
1 & $y^2 = x^7 -x^5 +x^3 +x $\\\vspace{0,1cm}
2 & $y^2 = x^8 + 7x^7 + 16x^6 + 10x^5 - 6x^4 - 2x^3 + 3x^2 - x$\\
\vspace{0,1cm}
3 &  $y^3 = x^4 -x +1$ \\\vspace{0,1cm}
7 & $ 2x^3+x^2y^2+x^2y+2xy+x+1=xy^2+y^4+2y^3+2y$\\
\bottomrule

\end{tabular}
\bigskip

\caption{Genus 3 curves $C$ over $\Q$ such that $\End^0 (\Jac(C_\Qbar))\simeq \Q(\sqrt{-d})$.}\label{table: examples}
\end{table}


\begin{thebibliography}{McK-Sta}

\bibitem[ACV20]{ACV20} J. Achter, S. Casalaina, and C. Vial,
\emph{Distinguished models of intermediate Jacobians}, \emph{J. Inst. Math. Jussieu} \textbf{19} (2020), no. 3, 891--918.

\bibitem[AFP22]{AFP22} S. Asif, F. Fité, and D. Pentland, 
\emph{Computing L-polynomials of Picard curves from Cartier-Manin matrices. With an appendix by A. V. Sutherland}, Math. Comp. \textbf{91} (2022), no. 334, 943--971.

\bibitem[BH04]{BH04} C. Birkenhake and H. Lange, \emph{Complex abelian varieties}, volume \textbf{302} (second edition) of Grundlehren der mathematischen Wissenschaften, Springer-Verlag, Berlin, 2004.

\bibitem[Bro24]{Bro} M. Brou\'e, \emph{From rings and modules to Hopf algebras--one flew over the algebraist's nest}, Springer, Cham, 2024.

\bibitem[CMSV19]{Rigorous_comp} E. Costa, N. Mascot, J. Sijsling and J. Voight, 
\emph{Rigorous computation of the endomorphism ring of a {J}acobian}, Math. Comp. \textbf{88} (2019), no. 317, 1303--1339.

\bibitem[Del79]{Del79} P. Deligne, \emph{Variétés de Shimura: interprétation modulaire, et techniques de construction de modèles canoniques}. In \emph{Automorphic forms, representations and $L$-functions (Part 2)}, Proc. Sympos. Pure Math., XXXIII, pages 247–289. Amer. Math. Soc., Providence, R.I., 1979.

\bibitem[Fal83]{Fal83} G. Faltings, \emph{Endlichkeitss\"atze f\"ur abelsche Variet\"aten \"uber Zahlk\"orpern}, Invent. Math. \textbf{73} (1983), 349--366.

\bibitem[FH91]{FH91} W. Fulton and J. Harris, \emph{Representation Theory: A First Course}, Graduate Texts in Math., Springer, New York, 1991.

\bibitem[FG22]{FG22} F. Fit\'e and X. Guitart, \emph{Tate module tensor decompositions and the Sato-Tate conjecture for certain abelian varieties potentially of $\GL_2$-type}, Mathematische Zeitschrift \textbf{300}, 2975--2995 (2022).

\bibitem[Fit24]{Fit24} F. Fit\'e, \emph{Ordinary primes for some varieties with extra endomorphisms}, Pub. Mat. \textbf{68} (2024), 27--40.


\bibitem[FKS25]{FKS25} F. Fit\'e, K.S. Kedlaya, and A.V. Sutherland, \emph{Sato-Tate groups of abelian threefolds}, Mem. Amer. Math. Soc. \textbf{312} (2025). 	

\bibitem[FS16]{FS16} F. Fit\'e and A. Sutherland, \emph{Sato-{T}ate groups of {$y^2=x^8+c$} and {$y^2=x^7-cx$}}, volume \textbf{663} of Contemp. Math., 103--126 (2016).

\bibitem[Hec59]{Hec59} E. Hecke, \emph{Mathernatische Werke}, Vandenhoeck und Ruprecht, Göttingen, 1959.

\bibitem[Hei34]{Heilbronn} H. Heilbronn \emph{On the class-number in imaginary quadratic fields}, The Quarterly Journal of Mathematics, Volume 5, Issue 1, 150-–160 (1934).

\bibitem[Kah21]{Kah21}
B. Kahn, \emph{Albanese kernels and Griffiths groups}, Tunisian Journal of Mathematics, Vol. 3 (2021), No. 3, 589--656. 

\bibitem[K\i l16]{Kil16}
P. K\i l\i \c{c}er, The CM class number one problem for curves,
PhD Thesis, Universit\'e de Bordeaux, 2016, https://tel.archives-ouvertes.fr/tel-01383309v2.

\bibitem[KR14]{KR14}
S. Kudla and M. Rapoport, \emph{Special cycles on unitary Shimura varieties II: Global theory}, J. Reine Angew. Math. \textbf{697} (2014), 91--157.

\bibitem[KS08]{KS08} K.S. Kedlaya and A.V. Sutherland, \emph{Computing $L$-series of hyperelliptic curves}, in: A.J. van der Poorten, A. Stein (eds.)
Algorithmic Number Theory: 8th International Symposium, ANTS-VIII (Banff, Canada, May 2008), 312--326. Lec. Notes Comp. Sci. vol. 5011. Springer, New York, 2008.

\bibitem[Lan83]{Lan83} S. Lang, \emph{Complex multiplication}, volume \textbf{255} of Grundlehren der mathematischen Wissenschaften, Springer-Verlag, New York, 1983.

\bibitem[LS24]{LS24} J. Laga and A. Shnidman, \emph{Vanishing criteria for Ceresa cycles}, Compos. Math. \textbf{161} (2025), no. 11, 3017–3043.

\bibitem[Mil72]{Mil72} J. Milne, \emph{On the arithmetic of abelian varieties}, Inventiones Mathematicae \textbf{17}, 177--190 (1972).

\bibitem[Mil86]{Mil86}
J. Milne, \emph{Abelian varieties}, in: G. Cornell, J.H. Silverman (eds.), Arithmetic geometry (Storrs, Conn., 1984), 103--150, Springer-Verlag, New York, 1986.

\bibitem[Mum83]{Mum} D. Mumford, \emph{Abelian varieties}, volume \textbf{5} of Tata Institute of Fundamental Research Studies in Mathematic, Oxford University Press, London, 1983.

\bibitem[Neu92]{Neu92} J. Neukirch, \emph{Algebraische Zahlentheorie}, Springer-Verlag Berlin-Heidelberg-New York, 1992.

\bibitem[OS18]{OS18} M. Orr and A.N. Skorobogatov, \emph{Finiteness theorems for K3 surfaces and abelian varieties of CM type}, Compositio Math. \textbf{154} (2018) 1571--1592.

\bibitem[OSZ21]{OSZ21} M. Orr, A.N. Skorobogatov, and Y. G. Zarhin, \emph{On uniformity conjectures for abelian varieties and K3 surfaces}, American Journal of Mathematics, \textbf{143} (6), 1665--1702 (2021). 

\bibitem[Rem17]{Rem17} G. R\'emond, \emph{Vari\'et\'es ab\'eliennes et ordres maximaux}, Rev. Mat. Iberoam. \textbf{33} (2017), no. 4, 1173--1195.

\bibitem[Rib76]{Rib76}
K. A. Ribet, \emph{Galois action on division points on abelian varieties with many real multiplications}, Am. J. Math. \textbf{98} (1976), 751--804.

\bibitem[Rib77]{Rib77} K. Ribet, \emph{Galois representations attached to eigenforms with Nebentypus}, in: J.-P. Serre, D. B. Zagier (eds.), Modular Functions of one Variable V (Bonn 1976), Lect. Notes in Math. 601, Springer (1977), 17--52.

\bibitem[Ser89]{Ser89} J.-P. Serre, \emph{Abelian $\ell$-adic representations and elliptic curves}, Addison-Wesley Publ. Co., Redding, Mass., 1989.

\bibitem[Ser12]{Ser12}
J.-P. Serre, \emph{Lectures on $N_X(p)$}, CRC Press, Boca Raton, FL, 2012.

\bibitem[Ser13]{Ser13} J.-P. Serre, \emph{Lettre à Marie-France Vignéras du 10/2/1986}, in: Oeuvres/Collected papers. IV. 1985--1998, Springer Collected Works in Mathematics, Springer, Heidelberg, 2013, 

\bibitem[Sch88]{Sch88} N. Schappacher \emph{Periods of {H}ecke characters}, volume \textbf{1301} of Lecture Notes in Mathematics, Springer-Verlag, Berlin, 1988.

\bibitem[Shi63]{Shi63} G. Shimura, \emph{On analytic families of polarized abelian varieties and automorphic functions}, volume \textbf{78} of Ann. of Math. (2), 149 -- 192, 1963.

\bibitem[Shi71]{Shi71} G. Shimura, \emph{On elliptic curves with complex multiplication as factors of the Jacobians of modular function fields}, Nagoya Math. J., \textbf{43} (1971), 99--208.

\bibitem[Shi98]{Shi98} G. Shimura, \emph{Abelian varieties with complex multiplication and modular functions}, Princeton University Press, Princeton, NJ, 1998.

\bibitem[Sil94]{Sil94} J. H. Silverman, \emph{Advanced topics in the arithmetic of elliptic curves}, volume \textbf{151} of Graduate Texts in Mathematics, Springer-Verlag, New York, 1994.


\bibitem[Sut19]{Sut19} A.V. Sutherland, \emph{A database of nonhyperelliptic genus 3 curves over $\Q$}, Thirteenth Algorithmic Number Theory Symposium (ANTS XIII), The Open Book Series 2 (2019), 443--459.

\bibitem[Wei56]{Wei56} A. Weil, \emph{On a certain type of characters of the idèle-class group of an algebraic number field}, in: Proceedings of the International Symposium on Algebraic Number Theory, Tokyo \& Nikko,
1955, Science Council of Japan, Tokyo, 1956, pp. 1--7.

\bibitem[Zar09]{Zarhin_superelliptic} Y. Zarhin, \emph{Endomorphisms of superelliptic {J}acobians}, Mathematische Zeitschrift, \textbf{261}, 691--707 (2009). See also arxiv.org/pdf/math/0605028.

\end{thebibliography}
\end{document}